\documentclass[12pt]{article}
\usepackage[utf8]{inputenc}
\usepackage{graphicx}

\usepackage{booktabs}
\usepackage{array}
\usepackage{paralist}
\usepackage{verbatim}
\usepackage{subfig} 
\usepackage{amsmath}
\usepackage{amssymb}
\usepackage{amsthm}
\usepackage[vcentermath]{youngtab}
\usepackage{thmtools}
\usepackage{dsfont}

\makeatletter
\def\th@plain{
  \thm@notefont{}
  \itshape
}
\def\th@definition{
  \thm@notefont{}
  \normalfont
}
\makeatother

\newcommand{\la}{\lambda}

\usepackage{fancyhdr} 
\title{Integer Compositions Applied to the Probability Analysis of Blackjack and the Infinite Deck Assumption}
\author{Jonathan Marino and David G. Taylor}
\date{} 

\begin{document}
\maketitle

\begin{abstract}
Composition theory can be used to analyze and enumerate the number of ways a dealer in Blackjack can reach any given point total. The rules of Blackjack provide several restrictions on the number of compositions of a given number.  While theory guarantees a specific number of unrestricted compositions of any positive integer, we must subtract the number of compositions not allowed in Blackjack.  We present a constructive approach to enumerate the number of possible compositions for any point value by deleting those illegal compositions from the total number of unrestricted compositions. Our results cover all possible cases and also generalize to changes to the rules of Blackjack, such as the point value where the dealer must stand.  Using the infinite deck assumption, we also find the approximate probability that the dealer reaches that point total.  
\end{abstract}


\section{Introduction}

In this paper, we explore Blackjack in terms of restricted compositions of the natural numbers.  The rules of Blackjack impose several interesting restrictions on compositions of numbers including how to score an Ace as either one or eleven.  We provide two ways to enumerate such compositions. One case is a closed form, useful in calculating the number of ways the dealer can reach a number within 11 of the his or her face up card, and the other is a general form to calculate all possible situations.  Although the closed form is sufficient in most situations, the general form abstracts all of the rules of Blackjack, including the range of card values, the rule where the dealer must stand, and the target total of points.

This pattern was the result of counting the number of ways the dealer could reach a total of 17 points with a face up card of 10 explicitly.  If the dealer drew no more cards, there was only one way to reach 17, that is, the dealer's face down card was a 10: $10+7$.  If the dealer drew one more card beyond the face up card and face down cards he or she currently has, there were five ways,\begin{align*}&10+2+5 \\ &10+3+4 \\ &10+4+3 \\ &10+5+2 \\ &10+6+1. \end{align*}  Continuing to enumerate the possibilities, we have the following pattern \begin{center}
\begin{tabular}{l | *{6}{c}}
Number of cards & 0 & 1 & 2 & 3 & 4 & 5 \\
\hline
Ways to reach 17 & 1 & 5 & 10 & 10 & 5 & 1

\end{tabular}
\end{center}
which is fifth row of Pascal's Triangle and thus the binomial coefficients for $n=5$.  We explore why this pattern occurs and generalize our findings.


\section{Partitions and Compositions}


\theoremstyle{definition}
\newtheorem{comp}{Definition}[section]
\begin{comp}

\cite{textbook} Let $c(m,n)$ denote the \emph{number of compositions} of $n$ with exactly $m$ parts.  That is, $c(m,n)$ is the number of integer partitions of $n$ with $m$ parts of the form $$n=\sum_{i=1}^{m}  \lambda_{i} \text{ where  } \lambda_i \in \mathds{Z}^+$$ where order matters. We call each $\la_i$ a \emph{part} of the composition, and let $| \la | = n$ be the sum of all the parts.  For a composition of $n$ with $m$ parts, we may write the composition as $\la_1+\la_2+\cdots+\la_m$ or $(\la_1,\la_2,\dots,\la_m)$.
Further, it has been shown that $$c(m,n)={n-1 \choose m-1}.$$

\end{comp}

The ability to enumerate compositions as binomial coefficients is key to proving the aforementioned pattern as fact.  To aide in counting the restricted compositions, we introduce the following notation.


\theoremstyle{definition}
\newtheorem{caplamb}[comp]{Definition}
\begin{caplamb}\label{caplamb}

Let $\Lambda(m,n)$ be the set of all compositions of $n$ with $m$ parts.  That is, $$\Lambda(m,n) = \left\{ (\la_1,\la_2,\dots,\la_m) : |\la|=n \right\}. $$  Also, it is important to note that $$|\Lambda(m,n)|=c(m,n)={n-1 \choose m-1}.$$
\end{caplamb}  

It is also helpful to visualize partitions and compositions as boxes in the form of Young tableau.


\theoremstyle{definition}
\newtheorem{youngt}[comp]{Example}
\begin{youngt}

Let $(3,2,4,1)$ be a composition of of 10 with 4 parts.  Then the \emph{Young tableau} of the composition is $$\yng(3,2,4,1)$$ where we note that the top 3 blocks represents $\la_1=3$, the second row of blocks represents $\la_2=2$, and so on.  There are total of $n=10$ blocks, and there are $m=4$ rows corresponding to the 4 parts of the composition.

\end{youngt}

Those familiar with partition theory will recognize Young tableaux.  Partitions are compositions with the restriction that $\la_i \ge \la_j$ if $i<j$, that is, for a partition of $n$ with $m$ parts, we have that $\la_1 + \la_2 + \cdots + \la_m=n$ and $\la_1 \ge \la_2 \ge \cdots \ge \la_m$.  Young tableaux for partitions have the same typical structure.  For example, consider the partition $4+3+2+1+1$.  Then we have $$\yng(4,3,2,1,1)$$ for the Young tableau representing $(4,3,2,1,1)$.  The Young tableau for any given partition will look like an ``upside down staircase'' due to the restriction that the parts be listed in descending order.

These definitions, along with the following lemma, can be used to help describe our pattern.


\theoremstyle{plain}
\newtheorem{pas}[comp]{Lemma}
\begin{pas}[Variation on Pascal's Rule]\label{pasprop}

For any $n \in \mathds{Z}^+$, we have $${n \choose k}-{n-1 \choose k-1}={n-1 \choose k} .$$

\end{pas}

\begin{proof}

The lemma follows directly from a manipulation of Pascal's Rule, $${n-1 \choose k}+{n-1 \choose k-1} = {n \choose k}.$$

\end{proof}

\section{Rules and Restrictions of Blackjack}


Blackjack has several interesting rules that restrict the number of possibilities of reaching a given point total. Each rule provides a restriction on the total number of compositions that we must remove from the total number, $c(m,n)$.  We will start with the total number of compositions of a given number and remove those compositions that are not allowed by the rules of Blackjack.

Since each card has a point value associated with it between 1 and 11, each part of the composition must be less than or equal to 11.  Restriction 1 removes compositions with pieces too big to be represented by playing cards.


\theoremstyle{definition}
\newtheorem{partsrule}{Restriction}

\begin{partsrule}
For every composition of point values, each part, $\la_i$ must be less than or equal to $A$, the highest valued card.  Note that in standard play, $A=11$.

\end{partsrule}

Next, it is usually the case that the first unknown card is not worth one, since if the first unknown is an ace, it would be worth 11.  In general, an Ace can only be worth 1 point if the dealer already has more than 10 points, causing the dealer to bust if the Ace is counted as 11. 


\newtheorem{rule1}[partsrule]{Restriction}
\begin{rule1}

The first part of the composition usually cannot be a one, that is, the composition cannot take the form $$1+\la_2+\la_3+\cdots+\la_m.$$

\end{rule1}

If the first unknown is an ace, it would be worth 11.  In general, an Ace can only be worth 1 point if the dealer already has more than 10 points, causing the dealer to bust if the Ace were counted as 11.  In a few rare cases, the first card could be a 1.  For example, consider the situation where a dealer with a face up two trying to reach 19.  Then we allow $(1,3,10,3)$ since $$2 + (1 + 3 + 10 + 3) = 19.$$  However, this can only happen when the difference between the dealer's goal and his or her face up card is greater than 11.  We explore this in detail in the general case of the resulting theorems.

Additionally, the dealer must stand at a certain point value.  In standard play, the dealer is required to stop once they reach 17.  If the dealer is trying to reach 18, the last card cannot be worth 1, since the dealer will reach 17 in $m-1$ steps.  Similarly, if the dealer is trying to reach 19, the last card cannot be worth 1 or 2, as the dealer would have to stand at 18 or 17, respectively, and be unable to draw additional cards.


\newtheorem{rulelast}[partsrule]{Restriction}
\begin{rulelast}

The last part of the composition cannot be any positive integer less than or equal to the difference between the desired number of points, $w$, and the point value where the dealer must stand, $s$, that is, the composition cannot take the form $$\la_1+\la_2+\la_3+\cdots+\la_{m-1}+y$$ where $y \in \mathds{Z} $ such that $1\le y \le w-s.$

\end{rulelast}

The previous two restrictions have dealt with Aces at the beginning and end of compositions.  Aces can be worth 1 or 11, but the dealer defaults to treating an Ace as an 11.  He or she continues to count the Ace as an 11 unless future cards cause the dealer to go over 21;  this situation allows the dealer to change the Ace to a 1 and reevaluate the current point total.  To account for the behavior of Aces in the middle of our compositions, we introduce the following definition.


\newtheorem{lambdaset}[comp]{Definition}
\begin{lambdaset}\label{i-ace}

If $d$ is the point value of the dealer's face up card and $s$ is the point value where the dealer must stand, then we define the \emph{i-Ace set} to be the subset of all compositions of $s-d$ consisting of all compositions of $10-d$ of length $i$.  That is, the set of compositions of length $i$ of the form $$\tau_i =\left\{ (\la_1,\la_2,\dots,\la_i) :  |\la| \le 10-d, \la_j \ge 2  \right \}. $$

\end{lambdaset}

Plainly, each $i$-Ace set represents compositions of $s-d$ of length $m > i+1$ where part $i+1$ is a 1 in situations where the 1 must be an 11. The restriction that each part be greater than 2 is to make sure that the $\la_{i+1}=1$ in question is indeed the first 1 in the composition.  Symbolically,  these are compositions of the form $$\la_1+\la_2+\cdots+\la_i+1+\la_{i+2}\cdots+\la_m = s-d$$  where $d +\la_1+ \cdots + \la_i \le 10$.  This forces the 1 to be an 11. While these are perfectly allowable compositions of $s-d$, they are not allowed in Blackjack.


\newtheorem{exset}[comp]{Example}
\begin{exset}

Consider the situation where the player is against the dealer, who has a face up 2 ($d=2$) and must stand at 17 ($s=17$). The player is interested in knowing how many ways the dealer can reach 17 with three additional cards ($m=3$), including the unknown face down card.  We are considering compositions of 15 of length 3. Then the 1-Ace set would be the set $$\tau_1= \{(\la_1,\la_2,\la_3) : \la_2=1, \la_{i\ne 2}\ge 2\}.$$  This represents the compositions $$
\begin{aligned}
2&+1+12 \\
3&+1+11 \\
4&+1+10 \\
5&+1+9 \\
6&+1+8 \\
7&+1+7 \\
8&+1+6
\end{aligned}$$
Note that the first part of each composition is of the form $\la_1 \le 8=10-2$.  While these are all compositions of 15, we must delete them from our total number of legal Blackjack compositions since, in these situations, it would be impossible that the second card is worth 1 since the dealer would be forced to consider the Ace an 11.  However, the composition $3+11+1$ is allowable and not counted in the 1-Ace set.
\end{exset}

Since the $i$-Ace sets are placeholders for certain illegal compositions, we must remove them from our total.


\newtheorem{acerule}[partsrule]{Restriction}
\begin{acerule}

The compositions represented by the $i$-Ace sets must be deleted.

\end{acerule}

However, if the composition in question is small enough, we need not concern ourselves with the \emph{i}-Ace sets.  In proving the theorem that explains our pattern by relating the number of allowable compositions to binomial coefficients, the following lemma will be useful.


\newtheorem{rmk9}[comp]{Lemma}
\begin{rmk9}\label{tauempty}

Let $w$ be the point total the dealer is trying to reach and $d$ be the dealer's face up card.  If the difference between the dealer's goal and face up card is less than 11, that is, $w-d\le 11$, then $\tau_i = \varnothing$ for all $i$.
\end{rmk9}

\begin{proof}

Let $w-d \le 11$.  Then $w \le 11+d$.  If $\tau_i \not= \varnothing$, then there exists some composition $\la_1+ \la_2 + \cdots + \la_i \le 10-d$.  Then we have, $d+\la_1+\cdots+\la_i+11 \le 21$.  But since each $\la_i \in \mathds{Z}^+$, $d+11 < d+\la_1+\cdots+\la_i+11$, and so $w \le 11+d < d+\la_1+\cdots+\la_i+11$, that is, our composition is larger than the target number of points.  Therefore that composition cannot exist since the Ace could not be worth 11 points, and so $\tau_i = \varnothing$.

\end{proof}

We now have the necessary tools to investigate our pattern further.


\section{Results}
If the dealer's face up card has a high enough value, the number of allowable compositions yields a closed form.


\theoremstyle{plain}
\newtheorem{thm}{Theorem}[section]
\begin{thm}\label{pretty}

Let $s$ be the the point value where the dealer must stand, $d$ be the point value of the dealer's face-up card, and $w$ be the desired number of points. If $w-d \le 11$, then the number of ways, $n$, the dealer can reach a total of $w$ points with $m$ cards (excluding the initial face up card) is $$n(m,s,d)=\binom{s-d-2}{m-1}.$$

\end{thm}

\begin{proof}
We are considering compositions of $w-d$.  The total number of compositions of $w-d$ with $m$ parts is given as $$c(m, w-d)={w-d-1 \choose m-1}$$ We will show that by applying the rules of Blackjack, we may remove exactly enough compositions to reach the stated number of compositions.  Let $R_n$ be the number of compositions removed by the $n^{th}$ restriction.  Then $$n(m,s,d)=c(m,w-d)-R_1-R_2-R_3-R_4.$$  Note that since $w-d\le 11$, no composition of $w-d$ can have parts $\la_i \ge 12$, and thus $R_1=0$.  Also, $\tau_i=\varnothing$ for $i=1,2,\dots,m-2$ by Lemma \ref{tauempty}.  So $R_4=0$.  

Consider Restriction 2.  Since the first card cannot be an Ace, we must subtract the number of compositions of $w-d$ that begin with 1, that is, any composition of the form $1+\la_2+\la_3+\cdots+\la_m$.  Note that if $1+\la_2+\la_3+\cdots+\la_m=w-d$, then $\la_2+\la_3+\cdots+\la_m=w-d-1$.  So these are compositions of $w-d-1$ with $m-1$ parts.  We can enumerate these compositions as $$R_2=c(m-1,w-d-1)={w-d-2 \choose m-2}.$$  

Now consider Restriction 3.  Since the dealer must stand on $s$, the last card cannot have a value of between 1 and $w-s$.  Thus $R_3$ must be the number of compositions of $w-d$ of the form $\la_1+\la_2+\cdots+\la_{m-1}+y$, where $y \in \mathds{Z} $ such that $1\le y \le w-s$.  Note that if $\la_1+\la_2+\cdots+\la_{m-1}+y=w-d$, then $\la_1+\la_2+\cdots+\la_{m-1}=w-d-y$.  So $R_3$ is the number of compositions of $w-d-y$ with $m-1$ parts.  Summing over all of the possibilities of $y$, we have that $$R_3 = \sum_{i=1}^{w-s} c(m-1,w-d-i)=\sum_{i=1}^{w-s} {w-d-i-1 \choose m-2}.$$  However, for each composition ending with $y$ we remove, we have already removed the composition $1+\cdots+y$, so we must add it back.  So if $1+\la_2+\la_3+\cdots+\la_{m-1}+y=w-d$, then $\la_2+\la_3+\cdots+\la_{m-1}=w-d-y-1$.  Let $R^*$ be the number of doubly removed compositions, that is, the number of compositions of $w-d-i-1$ with $m-2$ parts which we must add back. Summing over all of the possibilities of $y$, we have that $$R^*=\sum_{i=1}^{w-s} c(m-2,w-d-i-1)=\sum_{i=1}^{w-s} {w-d-i-2 \choose m-3}.$$  Combining these, we have an adjusted total for the number of compositions of $w-d$ with length $m$ allowed in Blackjack,
\begin{align*}
	n(m,s,d)= &\text{ }c(m,w-d)-R_2-R_3+R^* \\
	= & {w-d-1 \choose m-1}-{w-d-2 \choose m-2} \\
	&-\sum_{i=1}^{w-s} {w-d-i-1 \choose m-2}+\sum_{i=1}^{w-s} {w-d-i-2 \choose m-3} \\
	=&{w-d-1 \choose m-1}-{w-d-2 \choose m-2} \\
	&-\sum_{i=1}^{w-s} \left( {w-d-i-1 \choose m-2}-{w-d-i-2 \choose m-3} \right).
\end{align*}
Using Lemma \ref{pasprop}, we may combine the first two terms and simplify the summand, yielding $$n(m,s,d)={w-d-2 \choose m-1}-\sum_{i=1}^{w-s} {w-d-i-2 \choose m-2}. $$ Expanding the summation, for the right hand side, we have
\begin{align*}
{w-d-2 \choose m-1} &-{w-d-3 \choose m-2}- {w-d-4 \choose m-2}- \cdots \\ \cdots &- {w-d-(w-s-1)-2 \choose m-2}-{w-d-(w-s)-2 \choose m-2}.
\end{align*}We may again simplify the first two terms using Lemma \ref{pasprop}, equaling 
\begin{align*}
{w-d-3 \choose m-1}& -{w-d-4 \choose m-2}- {w-d-5 \choose m-2}- \cdots \\ \cdots &- {w-d-(w-s-1)-2 \choose m-2}-{w-d-(w-s)-2 \choose m-2}.
\end{align*}
Again, we may combine the first two terms.  With each simplification, the number of parts in the first coefficient is preserved while the next term is still offset by one. This allows us to collapse the sum.  We are left with $$ n(m,s,d)={w-d-(w-s-1)-2 \choose m-1}-{w-d-(w-s)-2 \choose m-2}.$$ Which simplifies to $$n(m,s,d)={w-d-(w-s)-2 \choose m-1}={s-d-2 \choose m-1}.$$

\end{proof}


\theoremstyle{definition}
\newtheorem{rm1}[thm]{Remark}
\begin{rm1}
The number of possible compositions to reach a total of $w$ does not depend on $w$ at all, but only the dealer's stand rule, $s$, and the dealer's face up card, $d$.  Since Theorem \ref{pretty} enumerates the number of compositions as a binomial coefficient, we see that as $m$ increases, we move along the $(s-d-2)$ row of Pascal's Triangle, agreeing with the aforementioned pattern.
\end{rm1} 


\newtheorem{ex1}[thm]{Example}
\begin{ex1}

Consider the situation of a player against a dealer that must stand at $s=17$ wondering how many ways the dealer can reach $w=18$ points with a face up Queen ($d=10$) without drawing any more cards.  Then we have $$n(1,17,10)={17-10-2 \choose 1-1}={5 \choose 0} = 1.$$  Representing the situation that the dealer has an 8 as their face down card.

\end{ex1}


\newtheorem{ex2}[thm]{Example}
\begin{ex2}
Now consider the situation where the dealer has a face up 9 ($d=9$).  Then the number of ways the dealer could reach a total of 19 points after revealing $m=2$ more cards is
$$n(3,17,9)={17-9-2 \choose 3-1}={6 \choose 1} = 6.$$ This represents the compositions \begin{align*}&9+2+8 \\ &9+3+7\\&9+4+6\\&9+5+5\\&9+6+4\\&9+7+3.\end{align*}  Note that the compositions $9+8+2$ and $9+9+1$ would cause the dealer to stand on 17 and 18, respectively, so they are removed from the total number of legal compositions.
\end{ex2}


\subsection{General Case}
The hypothesis of Theorem \ref{pretty} assumed $w-d \le 11$. We generalize the previous result to allow the dealer to have any face up card.

\theoremstyle{plain}
\newtheorem{generalthm}[thm]{Theorem}
\begin{generalthm}[General Case]\label{generalcase}

Let $s$ be the the point value where the dealer must stand, $d$ be the point value of the dealer's face-up card, $w$ be the desired number of points, $b$ be the value on which players bust, and $A$ be the value of the highest ranked card.  Then the number of ways, $n$, the dealer can reach a total of $w$ points with $m$ cards (excluding the initial face up card) is 

\begin{align*}
g(m,w,s,d)=  &\binom{s-d-2}{m-1}-  m\sum_{i=A+1}^{w-d} {w-d-i-1 \choose m-2} \\
&-\sum_{j=2}^{10-d} \sum_{i=0}^{m+3} {j-i-2 \choose i}{s-d-j \choose m-3-i} \\
&+\sum_{k=b-d-10}^{s-d-2} \sum_{i=1}^{m-2} {k-b+s-2 \choose i-1}\cdot {w-d-k-2\choose m-i-2}.
\end{align*}

\end{generalthm}


The extra terms become apparent when $w-d$ is no longer less than 11 since  $R_1$ and $R_4$ are no longer zero, and a few situations arise where the first card could be a 1.  Revisiting our example from earlier, when $w = 19$, $d = 2$, and $s = 17$, we allow $(1,3,10,3)$ since $$2 + (1 + 3 + 10 + 3) = 19.$$  In some situations, we see that Restriction 2 is removing compositions that should not be removed.  This leads us to revisit Restriction 2.

\theoremstyle{definition}
\newtheorem*{r2rev}{Restriction 2 (Revisited)} 
\begin{r2rev} In the general case, an Ace \textit{can} be the face-down card \textit{and} count as 1.
\end{r2rev}

It appears that Restriction 2 is removing too many compositions.  To compensate for this, we add terms removed from $R_2$ by noticing that any allowable 1 for an Ace as the face-down card has the form (for this example, $m = 6$) $$d + 1 + \underbrace{\lambda_2 + \lambda_3 + \lambda_4}_{\text{Piece 1}} \big| + \underbrace{\lambda_5 + \lambda_6}_{\text{Piece 2}} = w$$ where the vertical bar represents the tipping point upon which $$d + 1 + \lambda_2 < s$$ and $$d+11+\lambda_2+\lambda_3 > b$$ with $b$ representing the largest point total allowed before busting (normally 21).  This algebraically reduces to cases when $$b-d-11 < \lambda_2 + \lambda_3 < s-d-1.$$ We must add back these terms.  Let $R_2^*$ be the number of compositions removed by Restriction 2 that should not have been.  Then we must add back 

\begin{align*}
R_2^*&=\sum_{k=b-d-10}^{s-d-2} \sum_{i=1}^{m-2} \underbrace{c(i,k-(b-s-1))}_{\text{Piece 1}} \cdot \underbrace{c(m-i-1,w-(d+k+1))}_{\text{Piece 2}} \\
&= \sum_{k=b-d-10}^{s-d-2} \sum_{i=1}^{m-2} {k-b+s-2 \choose i-1}\cdot {w-d-k-2\choose m-i-2}
\end{align*} compositions to our total.  Note that this cannot happen when $w-d\leq 11$ so the previous theorem is unaffected.

To aid in proving Theorem \ref{generalcase}, we also define a way to enumerate the number of compositions of a given number with each part at least 2.  This will be helpful in enumerating the compositions represented by Restriction 4.


\theoremstyle{definition}
\newtheorem{chat}[thm]{Definition}
\begin{chat}\label{defchat}

Let $\hat{c}(m,n)$ be the compositions of $n$ with $m$ parts, with each $\la_i \ge 2$.

\end{chat} This definition is helpful in describing the kinds of compositions described by the $i$-Ace sets.  In order to be able to count how many compositions Restriction 4 removes, we establish how to enumerate such compositions.


\theoremstyle{plain}
\newtheorem{chatpf}[thm]{Lemma}
\begin{chatpf}[Enumerating $\hat{c}$] \label{chat}

The number of compositions of $n$ with $m$ parts all at least $2$ is the same as the number of compositions of $n-m$ with $m$ parts.  That is, $$\hat{c}(m,n)=c(m,n-m)={n-m-1 \choose m-1}.$$

\end{chatpf}

\begin{proof}

Let $\hat{\Lambda}(m,n)$ be the set of all compositions of $n$ with $m$ parts, each at least 2.   Define $\varphi:\hat{\Lambda}(m,n)\rightarrow \Lambda(m,n-m)$ by $\varphi(\la_i)= \la_i-1$. Then $\varphi$ takes each $\la_i$ from each composition in $\hat{\Lambda}$ to $\la_i-1$.  Then, since we know that $\varphi$ is a bijection (with inverse map $\varphi^{-1}(\la_i)=\la_i+1$), $|\hat{\Lambda}(m,n)| = |\Lambda(m,n-m)|$, and so $$\hat{c}(m,n)=c(m,n-m)={n-m-1 \choose m-1}.$$

\end{proof}

The following remark demonstrates the proof of Lemma \ref{chat} using Young tableaux.


\theoremstyle{definition}
\newtheorem{exchat}[thm]{Remark}
\begin{exchat}
Consider the composition $3+2+4+2+3=14$.  Note that $(3,2,4,2,3)\in\hat{\Lambda}(5,14)$. The Young tableau $$\yng(3,2,4,2,3)$$ represents the composition $(3,2,4,2,3)$.  Then $\varphi$ maps this composition to (2,1,3,1,2) by removing 1 from each piece, as seen in  $$\yng(3,2,4,2,3) \overset{\varphi}{\longrightarrow} \yng(2,1,3,1,2) $$  and hence chopping off the first column.  Note that the second composition sums to 9, which is $14-5$.  We describe this action as $$\varphi:\hat{\Lambda}(5,14)\rightarrow \Lambda(5,9).$$

\end{exchat}


We now have the necessary tools to complete the proof of Theorem \ref{generalcase}.

\begin{proof}[Proof of General Case]

We proceed in a constructive way as earlier.  For the total number of legal compositions, we have
\begin{align*}
g(m,w,s,d)&=c(m,w-d)-R_1-R_2-R_3-R_4+R^* + R_2^*\\
&=c(m,w-d)-R_2-R_3+R^*-R_1-R_4 + R_2^*\\
&= {s-d-2 \choose m-1}-R_1-R_4 + R_2^*
\end{align*}
by Theorem \ref{pretty}.  Consider Restriction 1.  We must remove all compositions of $w-d$ that have any $\la_i > A$.  The number of such compositions is equivalent to the number of compositions of $w-d-i$ as $i$ ranges from $A+1$ to $w-d$.  Considering the $m$ possibilities for each one, we have that
\begin{align*}
R_1&= m\sum_{i=A+1}^{w-d} c(m-1,w-d-i) \\
&= m\sum_{i=A+1}^{w-d} {w-d-i-1 \choose m-2}.
\end{align*}
Now consider Restriction 4.  We must remove all of the compositions in each of the $i$-Ace sets.  That is, we must remove
\begin{align*}
R_4&=\sum_{\la \in \tau_i} c(m-1-i, s-d -(|\la|+1)) \\
&= \sum_{\la \in \tau_i} {s-d-|\la|-2 \choose m-i-2}
\end{align*}
compositions from our total.  We may represent this as a double summation and expand.  Since $i$ ranges from $1$ to $m-2$, we have that $$R_4= \sum_{i=1}^{m-2} \sum_{\la \in \tau_i} {s-d-|\la|-2 \choose m-i-2}.$$ Expanding the outer sum, we have
\begin{align*}
\sum_{\la \in \tau_1} {s-d-|\la|-2 \choose m-3} &+ \sum_{\la \in \tau_2} {s-d-|\la|-2 \choose m-4} + \cdots  \\
&+ \cdots + \sum_{\la \in \tau_{m-2}} {s-d-|\la|-2 \choose 0}.
\end{align*} In each sum, we are summing over all of the elements of each $\tau_i$, which ranges from $1$ to at most $10-d$.  We manipulate the sums to obtain
$$
\begin{aligned}
\sum_{j=2}^{10-d}& \big| \{\la \in \tau_1 : |\la|=j\} \big| \cdot {s-d-j \choose m-3} + \cdots \\
&\cdots + \sum_{j=2}^{10-d}\big| \{\la \in \tau_{m-2} : |\la|=j\}  \big| \cdot {s-d-j \choose 0}.
\end{aligned}$$ In essence, we are considering the size of the different $i$-Ace sets.  For each set, we are looking at compositions of $j$ of length $i+1$ with each part $\la_k \ge 2$.  This is equivalent to $\hat{c}(i+1,j)$.  Using Lemma \ref{chat}, we change the summands to

$$\sum_{j=2}^{10-d} \sum_{i=0}^{m+3} \hat{c}(i+1,j)\cdot {s-d-j \choose m-3-i}.$$ And so, for Restriction 4, we have that $$R_4=\sum_{j=2}^{10-d} \sum_{i=0}^{m+3} {j-i-2 \choose i}{s-d-j \choose m-3-i}.$$  Combining our results with our revisitation of Restriction 2, we have that
 \begin{align*}
g(m,w,s,d)=&{s-d-2 \choose m-1}-R_1-R_4 + R_2^*\\
=&\binom{s-d-2}{m-1}-  m\sum_{i=12}^{w-d} {w-d-i-1 \choose m-2} \\
&-\sum_{j=2}^{10-d} \sum_{i=0}^{m+3} {j-i-2 \choose i}{s-d-j \choose m-3-i} \\
&+ \sum_{k=b-d-10}^{s-d-2} \sum_{i=1}^{m-2} {k-b+s-2 \choose i-1}\cdot {w-d-k-2\choose m-i-2}.
\end{align*}

\end{proof}


\theoremstyle{definition}
\newtheorem{exgen}[thm]{Example}
\begin{exgen}We are interested in knowing the number of ways the dealer can reach $w=18$ in $m=3$ cards with a face up card of $d=2$.  Plugging in to Theorem \ref{generalcase}, we have that

$$ \begin{aligned}
g(3,18,17,2) &=\binom{17-2-2}{3-1}-  3\sum_{i=12}^{18-2} {18-2-i-1 \choose 3-2} \\
&\qquad- \sum_{j=2}^{10-2} \sum_{i=0}^{3+3} {j-i-2 \choose i}{17-2-j \choose 3-3-i}  \\
&\qquad + \sum_{k=21-2-10}^{17-2-2} \sum_{i=1}^{3-2} {k-21+2-2 \choose i-1}\cdot {18-2-k-2\choose 3-i-2} \\
&= 187.
\end{aligned}$$

\end{exgen}


\section{Applications in Probability}
Now that we can enumerate how many ways the dealer can reach a certain point value, we can calculate the probability that he or she does.  For ease of calculations, we introduce the following theorem.


\theoremstyle{plain}
\newtheorem{infdeck}{Theorem}[section]
\begin{infdeck}[Infinite Deck Assumption] \emph{\cite{Dave}}
The infinite deck assumption in Blackjack fixes the probability of getting any non-ten-valued card as $4/52$ and getting a ten-valued card as $16/52$.
\end{infdeck}

From this point on, we will use this assumption.  This assumption introduces very little error in calculation, and can be used to determine strategies for games of Blackjack with a large number of decks.  Naturally, as the number of decks increases, the error decreases.  For a strategy chart for Blackjack with eight decks, only one out of the 290 decision boxes using the infinite deck assumption is wrong, and only by a negligible amount \cite{Dave}.

We combine Theorem \ref{pretty} with the Infinite Deck Assumption to provide the following proposition.


\theoremstyle{plain}
\newtheorem{probs}[infdeck]{Proposition}
\begin{probs}\label{probb}

Let $s$ be the the point value where the dealer must stand, $d$ be the point value of the dealer's face-up card, and $w$ be the desired number of points. If $w-d \le 9$, then the probability the dealer's final point total is $w$ is approximately $$p_w(s,d)=\frac{1}{13}\left(\frac{14}{13} \right )^{s-d-2}$$

\end{probs}

\begin{proof}

The dealer needs to reach a point value of $w$ in $s-d-1$ steps or less. Since $w-d \le 9$, we are only considering cards worth less than 10.  If they take one additional card, there are $n(1,s,d)$ possibilities by Theorem \ref{pretty}, each with a probability of $4/52=1/13$, since there are four suits of a card with a particular point value out of a total of 52 cards.  Then the probability of reaching $w$ with one additional card, $p_w^1$, is $$p_w^1=n(1,s,d)\cdot \left (\frac{1}{13} \right )={s-d-2 \choose 0}\left (\frac{1}{13} \right ).$$  Similarly, if the dealer were to take two cards in addition to the face up card, there are $n(2,s,d)$ possibilities, each with a probability of $\frac{1}{13} \cdot \frac{1}{13}$.  If we are considering the composition $(\la_1,\la_2)$, the probability of drawing $\la_1$ is 1/13 and the probability of drawing $\la_2$ is 1/13.  So the probability of reaching $w$ with two additional cards, $p_w^2$, is $$p_w^2={s-d-2 \choose 1}\left (\frac{1}{13} \right )\left (\frac{1}{13} \right )={s-d-2 \choose 1}\left (\frac{1}{13} \right )^2.$$In general, the probability that the dealer's final point total is $w$ after drawing $j$ cards is $$p_w^j={s-d-2 \choose j-1}\left (\frac{1}{13} \right )^j.$$  Since the dealer can reveal anywhere between one and $s-d-1$ additional cards, we sum over all the probabilities of the allowable compositions to reach $w$, $$\begin{aligned}\sum_{j=1}^{s-d-1} p_w^j & =\sum_{j=1}^{s-d-1} {s-d-2 \choose j-1}\left (\frac{1}{13} \right )^j \\&= \sum_{j=0}^{s-d-2} {s-d-2 \choose j}\left (\frac{1}{13} \right )^{j+1} \\&= \sum_{j=0}^{s-d-2} {s-d-2 \choose j}\left (\frac{1}{13} \right )^{j}\left (\frac{1}{13} \right ) \\&= \left (\frac{1}{13} \right )\sum_{j=0}^{s-d-2} {s-d-2 \choose j}\left (\frac{1}{13} \right )^{j}\end{aligned}$$
We may apply the Binomial Theorem to the summation, giving us 
\begin{align*}
\sum_{j=1}^{s-d-1} p_w^j&= \left ( \frac{1}{13} \right ) \left (1+\frac {1}{13} \right )^{s-d-2} \\
&=\frac{1}{13}\left(\frac{14}{13} \right )^{s-d-2}.
\end{align*}

\end{proof}

\theoremstyle{definition}
\newtheorem{probremark}[infdeck]{Remark}
\begin{probremark}
The restriction $w-d \le 9$ assures us that the probability of drawing a certain card is $1/13$ and that we may use the more attractive closed form of Theorem \ref{pretty}.
\end{probremark}

\newtheorem{exprob}[infdeck]{Example}
\begin{exprob}
Suppose a player is playing standard Blackjack against a dealer with a face up Jack ($d=10$).  The player is interested in knowing the probability that the dealer reaches 17.  Plugging in to Proposition \ref{probb} yields 
\begin{align*}
p_{17}(17,10)&=\frac{1}{13}\left(\frac{14}{13} \right )^{17-10-2} \\
&=\frac{1}{13}\left(\frac{14}{13} \right )^{5} \\
&\approx 0.1114.
\end{align*}
\end{exprob}

\theoremstyle{definition}
\newtheorem{exprob2}[infdeck]{Example}
\begin{exprob2}
Now suppose that the player has 17 points.  The dealer, after drawing several cards, has 12 points total.  The concerned player is interested in knowing the probability that the dealer beats him.  To calculate this, we consider the probability that the dealer ends with a final score between 18 and 21. 
\begin{align*}
\sum_{i=18}^{21} p_i(17,12) &= \sum_{i=18}^{21} \frac{1}{13}\left(\frac{14}{13} \right )^{3} \\
&=4 \cdot \frac{1}{13}\left(\frac{14}{13} \right )^{3} \\
&\approx 0.3843.
\end{align*}

\end{exprob2} 


\section{Conclusion}
We have provided a closed form to calculate the number of ways the dealer can reach a number within 11 of the his or her face up card.  Although this is sufficient in most situations, we also provide a general form to calculate all possible situations.  This general form abstracts all of the rules of Blackjack, including the range of card values, the rule where the dealer must stand, and the target total of points.  
\subsection{Future Goals}
In the future, we wish to simplify Theorem \ref{generalcase} to make it into a closed form resembling Theorem \ref{pretty}, allowing us to expand Proposition \ref{probb} to include all cases.


\end{document}